\documentclass [12pt]{amsart}
\usepackage{graphicx}
\usepackage{geometry}
\usepackage{amsthm}
\usepackage{amssymb}

\theoremstyle{plain}
\newtheorem{thm}{Theorem}[section]
\newtheorem{cor}[thm]{Corollary}
\newtheorem{lem}[thm]{Lemma}
\newtheorem{prop}[thm]{Proposition}
\newtheorem{defn}[thm]{Definition}
\newtheorem{exa}[thm]{Example}
\newtheorem{rem}[thm]{Remark}

\begin{document}

\title [{{On Graded Strongly $1$-Absorbing Primary Ideals}}]{On Graded Strongly $1$-Absorbing Primary Ideals}

 \author[{{R. Abu-Dawwas }}]{\textit{Rashid Abu-Dawwas }}

\address
{\textit{Rashid Abu-Dawwas, Department of Mathematics, Yarmouk University, Irbid, Jordan.}}
\bigskip
{\email{\textit{rrashid@yu.edu.jo}}}


\date{}

\begin{abstract} Let $G$ be a group with identity $e$ and $R$ be a $G$-graded commutative ring with nonzero unity $1$. In this article, we introduce the concept of graded strongly $1$-absorbing primary ideals. A proper graded ideal $P$ of $R$ is said to be a graded strongly $1$-absorbing primary ideal of $R$ if whenever nonunit homogeneous elements $x, y, z\in R$ such that $xyz\in P$, then either $xy\in P$ or $z\in Grad(\{0\})$ (the graded radical of $\{0\}$). Several properties of graded strongly $1$-absorbing primary ideals are investigated. Many results are given to disclose the relations between this new concept and others that already exist. Namely, the graded prime ideals, the graded primary ideals, and the graded $1$-absorbing primary ideals.
\end{abstract}

 \maketitle

\section{Introduction}

Throughout this article, all rings are commutative with a nonzero unity $1$. Let $G$ be a group with identity $e$. Then a ring $R$ is said to be $G$-graded if $R=\displaystyle\bigoplus_{g\in G} R_{g}$ with $R_{g}R_{h}\subseteq R_{gh}$ for all $g, h\in G$, where $R_{g}$ is an additive subgroup of $R$ for all $g\in G$. The elements of $R_{g}$ are called homogeneous of degree $g$. If $x\in R$, then $x$ can be written uniquely as $\displaystyle\sum_{g\in G}x_{g}$, where $x_{g}$ is the component of $x$ in $R_{g}$. Also, we set $h(R)=\displaystyle\bigcup_{g\in G}R_{g}$. Moreover, it has been proved in \cite{Nastasescue} that $R_{e}$ is a subring of $R$ and $1\in R_{e}$. Let $I$ be an ideal of a graded ring $R$. Then $I$ is said to be a graded ideal if $I=\displaystyle\bigoplus_{g\in G}(I\cap R_{g})$, i.e., for $x\in I$, $x=\displaystyle\sum_{g\in G}x_{g}$, where $x_{g}\in I$ for all $g\in G$. An ideal of a graded ring need not be graded (see \cite{Nastasescue}).

Let $I$ be a proper graded ideal of $R$. Then the graded radical of $I$ is $Grad(I)=\left\{x=\displaystyle\sum_{g\in G}x_{g}\in R:\mbox{ for all }g\in G,\mbox{ there exists }n_{g}\in \mathbb{N}\mbox{ such that }x_{g}^{n_{g}}\in I\right\}$. Note that $Grad(I)$ is always a graded ideal of $R$ (see \cite{Refai Hailat}).

\begin{prop}(\cite{Farzalipour})\label{1} Let $R$ be a $G$-graded ring.

\begin{enumerate}

\item If $I$ and $J$ are graded ideals of $R$, then $I+J$, $IJ$ and $I\bigcap J$ are graded ideals of $R$.

\item If $a\in h(R)$, then $Ra$ is a graded ideal of $R$.
\end{enumerate}
\end{prop}

Recall that a proper graded ideal $Q$ of $R$ is said to be graded primary if whenever $x,y\in h(R)$ with $xy\in Q$, then $x\in Q$ or $y\in Grad(Q)$. In this case $P =Grad(Q)$ is a graded prime ideal of $R$, and $Q$ is said to be graded $P$-primary. In \cite{Zoubi Sharafat}, Al-Zoubi and Sharafat introduced a generalization of graded primary ideals called graded $2$-absorbing primary ideals. A proper graded ideal $I$ of $R$ is called a graded $2$-absorbing primary ideal of $R$ if whenever $x, y, z\in h(R)$ and $xyz\in I$, then $xy\in I$ or $xz\in Grad(I)$ or $yz\in Grad(I)$. The concept of graded $2$-absorbing primary ideals is generalized in many ways (see for example \cite{Uregen}). Recently in \cite{Bataineh Dawwas}, we consider a new class of graded ideals called the class of graded $1$-absorbing primary ideals. A proper graded ideal $P$ of $R$ is said to be a graded $1$-absorbing primary ideal of $R$ if whenever nonunit elements $x, y, z\in h(R)$ such that $xyz\in P$, then $xy\in P$ or $z\in Grad(P)$. Clearly, every graded primary ideal is graded $1$-absorbing primary ideal. The next example shows that the converse is not true in general.

\begin{exa}\cite{Bataineh Dawwas} Let $K$ be a field and assume that $R=K[X,Y]$ is $\mathbb{Z}$-graded with $degX= 1 =degY$. Consider the graded ideal $P= (X^{2},XY)$ of $R$. Then $Grad(P)= (X)$. Since for $X.Y.X\in P$, either $X.Y\in P$ or $X\in Grad(P)$, $P$ is a graded $1$-absorbing primary ideal of $R$. On the other hand, $P$ is not graded primary ideal of $R$ by (\cite{Soheilnia Darani}, Example 2.11).
\end{exa}

Also, it is clear that every graded $1$-absorbing primary ideal is graded $2$-absorbing primary ideal. The next example shows that the converse is not true in general.

\begin{exa}\cite{Bataineh Dawwas} Let $R=\mathbb{Z}[i]$ and $G=\mathbb{Z}_{2}$. Then $R$ is $G$-graded by $R_{0}=\mathbb{Z}$ and $R_{1}=i\mathbb{Z}$. Consider $P=12R$. Then as $12\in h(R)$, $P$ is a graded ideal of $R$. By (\cite{Zoubi Sharafat}, Example 2.2 (ii)), $P$ is a graded $2$-absorbing primary ideal of $R$. On the other hand, $2, 3\in h(R)$ such that $2.2.3\in P$, but neither $2.2\in P$ nor $3\in Grad(P)$. So, $P$ is not graded $1$-absorbing primary ideal of $R$.
\end{exa}

In this article, we follow \cite{Al-Mahdi} to introduce and study a subclass of the class of graded $1$-absorbing primary ideals that does not necessarily contain all graded primary ideals. A proper graded ideal $P$ of $R$ is called graded strongly $1$-absorbing primary if whenever nonunit elements $x, y, z\in h(R)$ and $xyz\in P$, then $xy\in P$ or $z\in Grad(\{0\})$. Several properties of graded strongly $1$-absorbing primary ideals are investigated. Many results are given to disclose the relations between this new concept and others that already exist. Namely, the graded prime ideals, the graded primary ideals, and the graded $1$-absorbing primary ideals.

\section{Graded Strongly $1$-Absorbing Primary Ideals}

In this section, we introduce and study the concept of graded strongly $1$-absorbing primary ideals.

\begin{defn}Let $R$ be a graded ring and $P$ be a proper graded ideal of $R$. Then $P$ is said to be a graded strongly $1$-absorbing primary ideal of $R$ if whenever $x, y, z\in h(R)$ such that $xyz\in P$, then either $xy\in P$ or $z\in Grad(\{0\})$.
\end{defn}

Clearly, every graded strongly $1$-absorbing primary ideal is a graded $1$-absorbing primary ideal. However, the next example shows that the converse is not true in general. The same example shows that a graded prime ideal (in particular a graded primary ideal) is not necessarily a graded strongly $1$-absorbing primary ideal.

\begin{exa}\label{Example 2.1} Let $R=\mathbb{Z}[i]$ and $G=\mathbb{Z}_{2}$. Then $R$ is $G$-graded by $R_{0}=\mathbb{Z}$ and $R_{1}=i\mathbb{Z}$. Consider $P=3R$. Then as $3\in h(R)$, $P$ is a graded ideal of $R$. It is clear that $P$ is graded prime, and then graded $1$-absorbing primary. However, $P$ is not graded strongly 1-absorbing primary. Indeed, $2, 3\in h(R)$ with $2.2.3\in P$ but neither $2.2\in P$ nor $3\in Grad(\{0\})=\{0\}$.
\end{exa}

A graded ring is said to be graded local if it has one graded maximal ideal only. The next theorem gives a characterization of graded strongly $1$-absorbing primary ideals.

\begin{thm}\label{Theorem 2.2} Let $P$ be a proper graded ideal of $R$. Then $P$ is graded strongly $1$-absorbing primary if and only if
\begin{enumerate}
\item $P$ is graded $1$-absorbing primary and $Grad(P) =Grad(\{0\})$, or
\item $R$ is graded local with graded maximal ideal $X =Grad(P)$ and $X^{2}\subseteq P$.
\end{enumerate}
\end{thm}

\begin{proof} Suppose that $P$ is graded strongly $1$-absorbing primary. Then it is clear that $P$ is graded $1$-absorbing primary. Assume that $Grad(P)\neq Grad(\{0\})$. Assume that there exist nonunit elements $a, b\in h(R)$ such that $ab\notin P$. Let $x\in P$. Then since $P$ is a graded ideal, $x_{g}\in P$ for all $g\in G$, and then $abx_{g}\in P$ for all $g\in G$. Since $P$ is graded strongly $1$-absorbing primary, $x_{g}\in Grad(\{0\})$ for all $g\in G$, which implies that $x\in Grad(\{0\})$. So, $P\subseteq Grad(\{0\})$, and hence $Grad(P)=Grad(\{0\})$, which is a contradiction. Thus, $ab\in P$ for each nonunit elements $a, b\in h(R)$. Let $X$ be a graded maximal ideal of $R$. We have that $X^{2}\subseteq P$, and then $X =Grad(X^{2})\subseteq Grad(P)$. Thus,
$X =Grad(P)$ for each graded maximal ideal $X$ of $R$. Hence, we conclude that $R$ is graded local with graded maximal ideal $X =Grad(P)$ and $X^{2}\subseteq P$. Conversely, if $P$ is graded $1$-absorbing primary and $Grad(P) =Grad(\{0\})$, then $P$ is clearly a graded strongly $1$-absorbing primary ideal of $R$. Suppose that $R$ is graded local with graded maximal ideal $X =Grad(P)$ and $X^{2}\subseteq P$. Then for each nonunit elements $a, b\in h(R), ab\in X^{2}\subseteq P$, and then $P$ is trivially a graded strongly $1$-absorbing primary ideal of $R$.
\end{proof}

A graded commutative ring $R$ with unity is said to be a graded domain if $R$ has no homogeneous zero divisors. A graded commutative ring $R$ with unity is said to be a graded field if every nonzero homogeneous element of $R$ is unit. The next example shows that a graded field need not be a field.

\begin{exa}Let $R$ be a field and suppose that $F=\left\{x+uy:x, y\in R, u^{2}=1\right\}$. If $G=\mathbb{Z}_{2}$, then $F$ is $G$-graded by $F_{0}=R$ and $F_{1}=uR$. Let $a\in h(F)$ such that $a\neq0$. If $a\in F_{0}$, then $a\in R$ and since $R$ is a field, we have $a$ is a unit element. Suppose that $a\in F_{1}$. Then $a=uy$ for some $y\in R$. Since $a\neq0$, we have $y\neq0$, and since $R$ is a field, we have $y$ is a unit element, that is $zy=1$ for some $z\in R$. Thus, $uz\in F_{1}$ such that $(uz)a=uz(uy)=u^{2}(zy)=1.1=1$, which implies that $a$ is a unit element. Hence, $F$ is a graded field. On the other hand, $F$ is not a field since $1+u\in F-\{0\}$ is not a unit element since $(1+u)(1-u)=0$.
\end{exa}

\begin{rem}\label{Remark 2.3} Let $R$ be a graded local ring. If $P$ is a graded strongly $1$-absorbing primary ideal of $R$, then $Grad(P)$ needs not to be graded maximal. To see that, it suffices to consider any graded local domain which is not graded field. Then $\{0\} =Grad(\{0\})$ is graded strongly $1$-absorbing primary which is not graded maximal.
\end{rem}

\begin{cor}\label{Corollary 2.4} Let $P$ be a graded prime ideal of $R$. Then $P$ is graded strongly $1$-absorbing primary if and only if
\begin{enumerate}
\item $P =Grad(\{0\})$, or
\item $R$ is graded local with graded maximal ideal $P$.
\end{enumerate}
\end{cor}

\begin{lem}\label{Theorem 2}\cite{Bataineh Dawwas} Let $R$ be a graded ring and $P$ be a graded ideal of $R$. If $P$ is a graded $1$-absorbing primary ideal of $R$, then $Grad(P)$ is a graded prime ideal of $R$.
\end{lem}

\begin{proof}Let $a, b\in h(R)$ such that $ab\in Grad(P)$. We may assume that $a, b$ are nonunit elements of $R$. Let $k\geq2$ be an even positive integer such that $(ab)^{k}\in P$. Then $k = 2s$ for some positive integer $s\geq1$. Since $(ab)^{k} = a^{k}b^{k} = a^{s}a^{s}b^{k}\in P$ and $P$ is a graded $1$-absorbing primary ideal of $R$, we conclude that $a^{s}a^{s} = a^{k}\in P$ or $b^{k}\in P$. Hence, $a\in Grad(P)$ or $b\in Grad(P)$. Thus $Grad(P)$ is a graded prime ideal of $R$.
\end{proof}

\begin{thm}\label{Theorem 2.6} Let $R$ be a graded ring. Then there exists a graded strongly $1$-absorbing primary ideal of $R$ if and only if
\begin{enumerate}
\item $Grad(\{0\})$ is graded prime, or
\item $R$ is graded local.
\end{enumerate}
\end{thm}

\begin{proof} Suppose that $R$ contains a graded strongly $1$-absorbing primary ideal $P$. If $R$ is not graded local, then by Theorem \ref{Theorem 2.2}, $Grad(P) =Grad(\{0\})$. Now, by Lemma \ref{Theorem 2}, $Grad(P)$ is graded prime since $P$ is graded $1$-absorbing primary. Thus, $Grad(\{0\})$ is graded prime. Conversely, by Corollary \ref{Corollary 2.4}, if $(R, X)$ is graded local, then $X$ is a graded strongly $1$-absorbing primary ideal, and if $P =Grad(\{0\})$ is graded prime, then it is a graded strongly $1$-absorbing primary ideal.
\end{proof}

A graded ring $R$ is said to be graded Artinian if every descending chain of graded ideals of $R$ terminates.

\begin{cor}\label{Corollary 2.7} Let $n\geq 2$ be an integer and the ring $\mathbb{Z}/n\mathbb{Z}$ be $G$-graded by a group $G$. Then $\mathbb{Z}/n\mathbb{Z}$ has a graded strongly $1$-absorbing primary ideal if and only if $n = p^{m}$ for some prime $p$ and positive integer $m$.
\end{cor}

\begin{proof} By Theorem \ref{Theorem 2.6}, $\mathbb{Z}/n\mathbb{Z}$ has a graded strongly $1$-absorbing primary ideal if and only if it is graded local or $Grad(\{0\})$ is graded prime. On the other hand, $\mathbb{Z}/n\mathbb{Z}$ is Artinian and so it is graded Artinian, and then $\mathbb{Z}/n\mathbb{Z}$ is graded local if and only it is graded field, that is $n$ is prime number. Now, if $n = p_{1}^{k_{1}}...p_{r}^{k_{r}}$ (the primary decomposition of $n$), then $Grad(\{0\}) = p_{1}... p_{r}\mathbb{Z}/n\mathbb{Z}$.
Hence, $Grad(\{0\})$ is graded prime if and only if $r = 1$.
\end{proof}

Let $R$ and $S$ be two $G$-graded rings. Then $R\times S$ is $G$-graded ring by $(R\times S)_{g}=R_{g}\times S_{g}$ for all $g\in G$.

\begin{cor}\label{Corollary 2.8} Let $R$ and $S$ be two $G$-graded rings. Then $R\times S$ has no graded strongly $1$-absorbing primary ideal.
\end{cor}

\begin{proof} The result holds from the fact that $R\times S$ is not graded local and $Grad(\{0_{R\times S}\})=Grad(\{0_{R}\})\times Grad(\{0_{S}\})$ is never graded prime ideal in $R\times S$.
\end{proof}

\begin{prop}\label{Proposition 2.9} Let $P$ be a proper graded ideal of $R$. Then $P$ is a graded strongly $1$-absorbing primary ideal of $R$ if and only if whenever $IJK\subseteq P$ for some proper graded ideals $I, J$ and $K$ of $R$, then $IJ\subseteq P$ or $K\subseteq Grad(\{0\})$.
\end{prop}

\begin{proof} Suppose that $P$ is a graded strongly $1$-absorbing primary ideal of $R$. Assume that $I, J$ and $K$ be proper graded ideals of $R$ such that $IJK\subseteq P$ and $IJ\nsubseteq P$. Then there exist $a\in I$ and $b\in J$ such that $ab\notin P$. Since $a\in I$ and $I$ is a graded ideal, $a_{g}\in I$ for all $g\in G$. Similarly, $b_{g}\in J$ for all $g\in G$. Since $ab\notin P$, there exist $g, h\in G$ such that $a_{g}b_{h}\notin P$. Let $x\in K$. Then $x_{r}\in K$ for all $r\in G$, and then $a_{g}b_{h}x_{r}\in P$ for all $r\in G$. Since $P$ is graded strongly $1$-absorbing primary, $x_{r}\in Grad(\{0\})$ for all $r\in G$, and then $x\in Grad(\{0\})$. Hence, $K\subseteq Grad(\{0\})$. Conversely, let $x, y, z\in h(R)$ be nonunit elements such that $xyz\in P$. Then $I=Rx$, $J=Ry$ and $K=Rz$ are proper graded ideals of $R$ with $IJK\subseteq P$, and then $IJ\subseteq P$ or $K\subseteq Grad(\{0\})$, which implies that $xy\in P$ or $z\in Grad(\{0\})$. Hence, $P$ is a graded strongly $1$-absorbing primary ideal of $R$.
\end{proof}

\begin{prop}\label{Proposition 2.10} Let $R$ be a graded ring, $P$ and $K$ be two proper graded ideals of $R$. If $P$ and $K$ are graded strongly $1$-absorbing primary, then so is $P\bigcap K$.
\end{prop}

\begin{proof} By Proposition \ref{1}, $P\bigcap K$ is a graded ideal of $R$. Let $x, y, z\in h(R)$ be nonunit elements such that $xyz\in P\bigcap K$. Then $xyz\in P$ and $xyz\in K$. If $z\in Grad(\{0\})$, then it is done. Suppose that $z\notin Grad(\{0\})$. Since $xyz\in P$ and $P$ is graded strongly $1$-absorbing primary, $xy\in P$. Similarly, $xy\in K$. Hence, $xy\in P\bigcap K$, and thus $P\bigcap K$ is a graded strongly $1$-absorbing primary ideal of $R$.
\end{proof}

\begin{prop}\label{Proposition 2.11(1)} Let $R$ be a graded ring such that every element of $h(R)$ is either nilpotent or unit. Then $Ra$ is a graded strongly $1$-absorbing primary ideal of $R$ for all nonunit $a\in h(R)$.
\end{prop}

\begin{proof} Let $a\in h(R)$ be nonunit element. By Proposition \ref{1}, $Ra$ is a graded ideal of $R$, and as $a$ is nonunit, $Ra$ is proper. Assume that $x, y, z\in h(R)$ be nonunit elements such that $xyz\in Ra$. Then $z\in Grad(\{0\})$, and hence $Ra$ is a graded strongly $1$-absorbing primary ideal of $R$.
\end{proof}

\begin{cor}\label{Proposition 2.11(2)} Let $R$ be a graded ring such that every element of $h(R)$ is either nilpotent or unit. Then every proper graded ideal of $R$ is graded strongly $1$-absorbing primary.
\end{cor}

\begin{proof} Let $P$ be a proper graded ideal of $R$. Assume that $x, y, z\in h(R)$ be nonunit elements such that $xyz\in P$ and $z\notin Grad(\{0\})$. Then $xyz\in h(R)$ is nonunit element with $xyz\in R(xyz)$. By Proposition \ref{Proposition 2.11(1)}, $R(xyz)$ is graded strongly $1$-absorbing primary, and then $xy\in R(xyz)\subseteq P$. Hence, $P$ is a graded strongly $1$-absorbing primary ideal of $R$.
 \end{proof}

\begin{exa} Every element of $R=\mathbb{Z}/9\mathbb{Z}$ is either unit or nilpotent. So, if $R$ is $G$-graded by any group $G$, then every homogeneous element of $R$ is either unit or nilpotent, and then every proper graded ideal of $R$ is graded strongly $1$-absorbing primary by Corollary \ref{Proposition 2.11(2)}.
\end{exa}

\begin{prop}\label{Proposition 2.12} Let $R$ be a graded ring. Then every graded prime ideal of $R$ is graded strongly $1$-absorbing primary if and only if $R$ is graded local and has at most one graded prime ideal that is not graded maximal.
\end{prop}

\begin{proof} Suppose that every graded prime ideal of $R$ is graded strongly $1$-absorbing primary. Assume that $R$ is not graded local and let $P$ be a graded prime ideal of $R$. Then since $P$ is graded strongly $1$-absorbing primary, we get $Grad(\{0\}) = P$ by Corollary \ref{Corollary 2.4}. Hence, $Grad(\{0\})$ is the only graded prime ideal of $R$. Thus, $R$ is graded local, which is a contradiction. Hence, $R$ is graded local with graded maximal ideal $X$. Now, let $P$ be a graded prime ideal which is not graded maximal. Then $Grad(P)=Grad(\{0\})$. Hence, $R$ has at most two graded prime ideals $Grad(\{0\})$ and $X$. Conversely, if every homogeneous element of $R$ is either unit or nilpotent, then it is done by Corollary \ref{Proposition 2.11(2)}. Otherwise, $R$ has exactly two graded prime ideals $P\subsetneq X$. Since $Grad(\{0\})$ is the intersection of graded prime ideals, we get that $Grad(\{0\}) = P$. It is clear that both of
$Grad(\{0\})$ and $X$ are graded strongly $1$-absorbing primary ideals, as desired.
\end{proof}

\begin{prop}\label{Proposition 2.14} Let $R$ be a graded ring. Then every graded primary ideal of $R$ is graded strongly $1$-absorbing primary if and only if \begin{enumerate}
\item every element of $h(R)$ is either nilpotent or unit, or
\item $R$ is graded local with graded maximal ideal $X$, one graded prime ideal that is not graded maximal (which is $Grad(\{0\})$, and every graded $X$-primary ideal contains $X^{2}$.
\end{enumerate}
\end{prop}

\begin{proof}Suppose that every graded primary ideal of $R$ is graded strongly $1$-absorbing primary. If (1) does not hold, then by Proposition \ref{Proposition 2.12}, $R$ is graded local with exactly tow graded prime ideals which are $Grad(\{0\})$ and $X$ (the graded maximal ideal). Now, let $P$ be a graded $X$-primary ideal of $R$. Then since $Grad(\{0\})\neq Grad(P) = X$, we have necessarily $X^{2}\subseteq P$ (by Theorem \ref{Theorem 2.2}). Conversely, if (1) holds, then the result follows trivially. Hence, suppose that (2) holds. Let $P$ be a graded primary ideal of $R$. Then $Grad(P)=Grad(\{0\})$ or $Grad(P)=X$. If $Grad(P)=Grad(\{0\})$, then by Theorem \ref{Theorem 2.2}, $P$ is graded strongly $1$-absorbing primary since every graded primary is graded $1$-absorbing primary. Now, if $Grad(P)=X$, then also $P$ is graded strongly $1$-absorbing primary since $X^{2}\subseteq P$.
\end{proof}

\begin{prop}\label{Proposition 2.17} Let $R$ be a graded ring. Then $\{0\}$ is the only graded strongly $1$-absorbing primary ideal of $R$ if and only if
\begin{enumerate}
\item $R$ is a graded field, or
\item $R$ is a graded domain that is not graded local.
\end{enumerate}
\end{prop}

\begin{proof} Suppose that $\{0\}$ is the only graded strongly $1$-absorbing primary ideal of $R$. If $R$ is graded local with graded maximal ideal $X$, then $X = \{0\}$ since $X$ is a graded strongly $1$-absorbing primary ideal of $R$. Hence, $R$ is a graded field. Now, if $R$ is not graded local, then $Grad(\{0\}) = \{0\}$ is graded prime since $Grad(\{0\})$ is a graded strongly $1$-absorbing primary ideal. Thus, $R$ is a graded domain. The converse is obvious.
\end{proof}

\begin{lem}\label{2} Let $R$ be a graded ring. Suppose that $P$ and $K$ are graded ideals of $R$. Then $(P:K)=\left\{r\in R:rK\subseteq P\right\}$ is a graded ideal of $R$.
\end{lem}

\begin{proof} Clearly, $(P:K)$ is an ideal of $R$. Let $r\in (P:K)$. Then $rK\subseteq P$ and $r=\displaystyle\sum_{g\in G}r_{g}$. Let $x\in K$. Then since $K$ is graded, $x_{h}\in K$ for all $h\in G$. Now, $r_{g}x_{h}\in h(R)$ for all $g, h\in G$ with $\displaystyle\sum_{g\in G}r_{g}x_{h}=\left(\displaystyle\sum_{g\in G}r_{g}\right)x_{h}=rx_{h}\in rK\subseteq P$ for all $h\in G$. Since $P$ is graded, $r_{g}x_{h}\in P$ for all $g, h\in G$, and then $r_{g}x\in P$ for all $g\in G$, which implies that $r_{g}K\subseteq P$ for all $g\in G$, so $r_{g}\in (P:K)$ for all $g\in G$. Hence, $(P:K)$ is a graded ideal of $R$.
\end{proof}

\begin{lem}\label{Lemma 2.18} Let $P$ be a graded $1$-absorbing primary ideal of $R$ and $K\nsubseteq P$ be a proper graded ideal of $R$. Then $(P : K)$ is a graded primary ideal of $R$.
\end{lem}

\begin{proof} By Lemma \ref{2}, $(P:K)$ is a graded ideal of $R$ and it is proper as $K\nsubseteq P$. Let $x, y\in h(R)$ with $xy\in (P : K)$ and $x\notin (P : K)$. Clearly, $y$ is a nonunit element. If $x$ is unit, then $y\in (P : K)\subseteq Grad((P : K))$. Hence, we may assume that $x$ and $y$ are nonunit elements of $R$. Since $x\notin (P : K)$, there is $z\in K$ such that $xz\notin P$, and then $xz_{g}\notin P$ for some $g\in G$. Note that $z_{g}\in K$ as $K$ is graded. But $xyz_{g}\in P$ and $P$ is graded $1$-absorbing primary. Then $y\in Grad(P)\subseteq Grad((P : K))$. Thus, $(P : K)$ is a graded primary ideal of $R$.
\end{proof}

\begin{prop}\label{Proposition 2.19} Let $P$ be a graded strongly $1$-absorbing primary ideal of $R$ and $K\nsubseteq Grad(P)$ be a proper graded ideal of $R$. Then $(P : K)$ is a graded strongly $1$-absorbing primary ideal of $R$.
\end{prop}

\begin{proof} If $Grad(P)\neq Grad(\{0\})$, then $R$ is graded local with graded maximal ideal $Grad(P)$. In this situation, our assumption $K\nsubseteq Grad(P)$ is not satisfied. Then we must have that $Grad(P)=Grad(\{0\})$ is graded prime. Let $a\in (P : K)$. Then $a_{g}\in (P:K)$ for all $g\in G$ as $(P:K)$ is a graded ideal by Lemma \ref{2}, and then $a_{g}K\subseteq P\subseteq Grad(P)$ for all $g\in G$. Since $K\nsubseteq Grad(P)$, we get $a_{g}\in Grad(P)$ for all $g\in G$, which implies that $a\in Grad(P)$. Hence, $P\subseteq (P : K)\subseteq Grad(P)$. Thus, $Grad((P : K)) =Grad(P) =Grad(\{0\})$. Now, the result follows from Theorem \ref{Theorem 2.2} and Lemma \ref{Lemma 2.18}.
\end{proof}

Let $R$ and $S$ be two $G$-graded rings. Then a ring homomorphism $f:R\rightarrow S$ is said to be graded homomorphism if $f(R_{g})\subseteq S_{g}$ for all $g\in G$.

\begin{lem}\label{3} Suppose that $f:R\rightarrow S$ is a graded homomorphism.
\begin{enumerate}
\item If $K$ is a graded ideal of $S$, then $f^{-1}(K)$ is a graded ideal of $R$.
\item If $P$ is a graded ideal of $R$ with $Ker(f)\subseteq P$, then $f(P)$ is a graded ideal of $f(R)$.
\end{enumerate}
\end{lem}

\begin{proof}
\begin{enumerate}
\item Clearly, $f^{-1}(K)$ is a an ideal of $R$. Let $x\in f^{-1}(K)$. Then $x\in R$ with $f(x)\in K$, and $x=\displaystyle\sum_{g\in G}x_{g}$ where $x_{g}\in R_{g}$ for all $g\in G$. So, for every $g\in G$, $f(x_{g})\in f(R_{g})\subseteq S_{g}$ such that $\displaystyle\sum_{g\in G}f(x_{g})=f\left(\displaystyle\sum_{g\in G}x_{g}\right)=f(x)\in K$. Since $K$ is graded, $f(x_{g})\in K$ for all $g\in G$, i.e., $x_{g}\in f^{-1}(K)$ for all $g\in G$. Hence, $f^{-1}(K)$ is a graded ideal of $R$.
\item Clearly, $f(P)$ is an ideal of $f(R)$. Let $y\in f(P)$. Then $y\in f(R)$, and so there exists $x\in R$ such that $y=f(x)$. So $f(x)\in f(P)$, which implies that $x\in P$ since $Ker(f)\subseteq P$, and hence $x_{g}\in P$ for all $g\in G$ since $P$ is graded. Thus, $y_{g}=(f(x))_{g}=f(x_{g})\in f(P)$ for all $g\in G$. Therefore, $f(P)$ is a graded ideal of $f(R)$.
\end{enumerate}
\end{proof}

\begin{prop}\label{Proposition 3.1} Suppose that $f:R\rightarrow S$ is a graded homomorphism.
\begin{enumerate}
\item If $f$ is a graded epimorphism and $P$ is a graded strongly $1$-absorbing primary ideal of $R$ containing $Ker(f)$, then $f(P)$ is a graded strongly $1$-absorbing primary ideal of $S$.
\item If $f$ is a graded monomorphism and $K$ is a graded strongly $1$-absorbing primary ideal of $S$, then $f^{-1}(K)$ is a graded strongly $1$-absorbing primary ideal of $R$.
\end{enumerate}
\end{prop}

\begin{proof}
\begin{enumerate}
\item By Lemma \ref{3}, $f(P)$ is a graded ideal of $S$. Let $a, b, c\in h(S)$ be nonunit elements such that $abc\in f(P)$. Since $f$ is a graded epimorphism, there exist nonunit elements $x, y, z\in h(R)$ such that $a = f(x), b = f(y)$ and $c = f(z)$. Suppose that $ab\notin f(P)$. Then $xy\notin P$. Hence, since $P$ is a graded strongly $1$-absorbing primary ideal of $R$ and $xyz\in P$ (because $Ker(f)\subseteq P$), we get that $z\in Grad(\{0_{R}\})$. Thus, $c = f(z)\in Grad(\{0_{S}\})$. Consequently, $f(P)$ is a graded strongly $1$-absorbing primary ideal of $S$.

\item By Lemma \ref{3}, $f^{-1}(K)$ is a graded ideal of $R$. Let $a, b, c\in h(R)$ be nonunit elements such that $abc\in f^{-1}(K)$. Suppose that $ab\notin f^{-1}(K)$. Then $f(a)f(b)\notin K$. Hence, since $K$ is a graded strongly $1$-absorbing primary ideal of $S$ and $f(a)f(b)f(c)\in K$, we conclude that $f(c)\in Grad(\{0_{S}\})$. Thus, since $f$ is a graded monomorphism, $c\in Grad(\{0_{R}\})$. Consequently, $f^{-1}(K)$ is a graded strongly $1$-absorbing primary ideal of $R$.
\end{enumerate}
\end{proof}

Let $R$ be a $G$-graded ring and $K$ be a graded ideal of $R$. Then $R/K$ is $G$-graded by $(R/K)_{g}=(R_{g}+K)/K$ for all $g\in G$.

\begin{lem}\label{4} Let $R$ be a graded ring and $K$ be a graded ideal of $R$. If $P$ is a graded ideal of $R$ with $K\subseteq P$, then $P/K$ is a graded ideal of $R/K$.
\end{lem}

\begin{proof} Clearly, $P/K$ is an ideal of $R/K$. Let $x+K\in P/K$. Then $x\in P$, and then $x_{g}\in P$ for all $g\in G$ as $P$ is graded, which implies that $x_{g}+K\in P/K$ for all $g\in G$. Hence, $P/K$ is a graded ideal of $R/K$.
\end{proof}

\begin{cor}\label{Corollary 3.2 (1)} Let $K\subseteq P$ be two proper graded ideals of a graded ring $R$. If $P$ is a graded strongly $1$-absorbing primary ideal of $R$, then $P/K$ is a graded strongly $1$-absorbing primary ideal of $R/K$.
\end{cor}

\begin{proof} Define $f:R\rightarrow R/K$ by $f(x)=x+K$. Then $f$ is a graded epimorphism with $Ker(f)=K\subseteq P$. So, by Proposition \ref{Proposition 3.1}, $f(P)=P/K$ is a graded strongly $1$-absorbing primary ideal of $R/K$.
\end{proof}

\begin{cor}\label{Corollary 3.2 (2)} Let $P$ be a graded strongly $1$-absorbing primary ideal of $R$ and $S$ be a graded subring of $R$. Then $P\bigcap S$ is a graded strongly $1$-absorbing primary ideal of $S$.
\end{cor}

\begin{proof} Define $f:S\rightarrow R$ by $f(x)=x$. Then $f$ is a graded monomorphism. So, by Proposition \ref{Proposition 3.1}, $f^{-1}(P)=P\bigcap S$ is a graded strongly $1$-absorbing primary ideal of $S$.
\end{proof}

\begin{cor}\label{5} Let $P$ be a graded strongly $1$-absorbing primary ideal of $R$. Then $P_{e}$ is a strongly $1$-absorbing primary ideal of $R_{e}$.
\end{cor}

\begin{proof} Apply Corollary \ref{Corollary 3.2 (2)} with $S=R_{e}$.
\end{proof}

Let $S$ be a multiplicatively closed subset of $h(R)$. Then $S^{-1}R$ is a graded ring with $(S^{-1}R)_{g}=\left\{  \frac{a}{s},a\in R_{h}, s\in S\cap R_{hg^{-1}%
}\right\}$.

\begin{prop}\label{Proposition 3.3} Let $R$ be a graded ring, $S$ be a multiplicatively closed subset of $h(R)$ and $P$ is a graded strongly $1$-absorbing primary ideal of $R$ such that $P\bigcap S = \emptyset$. Then $S^{-1}P$ is a graded strongly $1$-absorbing primary ideal of $S^{-1}R$.
\end{prop}

\begin{proof} Since $P\bigcap S =\emptyset$, $S^{-1}P$ is a proper graded ideal of $S^{-1}R$. Let $\frac{x}{s_{1}}\frac{y}{s_{2}}\frac{z}{s_{3}}\in S^{-1}P$ for some nonunit elements $x, y, z\in h(R)$ and $s_{1}, s_{2}, s_{3}\in S$. Then there is $u\in S$ such that $uxyz\in P$. Suppose that $\frac{x}{s_{1}}\frac{y}{s_{2}}\notin S^{-1}P$. Then $uxy\notin P$. Hence, $z\in Grad(\{0_{R}\})$ since $P$ is a graded strongly $1$-absorbing primary ideal of $R$. Thus, $\frac{z}{s_{3}}\in S^{-1}Grad(\{0_{R}\}=Grad(\{0_{S^{-1}R}\})$.
\end{proof}

Let $R$ be a $G$-graded ring. Then $R[X]$ is $G$-graded by $\left(R[X]\right)_{g}=R_{g}[X]$ for all $g\in G$.

\begin{lem}\label{6} Let $R$ be a $G$-graded ring and $P$ be a graded ideal of $R$. Then $P[X]$ is a graded ideal of $R[X]$.
\end{lem}

\begin{proof}Let $f(X)\in P[X]$. Then $f(X) = a_{0} +a_{1}X +...+a_{n}X^{n}$ for some $a_{0}, a_{1},..., a_{n}\in P$. Since $P$ is a graded ideal, we have $(a_{k})_{g}\in P$ for all $k=0, 1,..., n$ and $g\in G$, and then $(f(X))_{g} = (a_{0})_{g} +(a_{1})_{g}X +...+(a_{n})_{g}X^{n}\in P[X]$ for all $g\in G$. Hence, $P[X]$ is a graded ideal of $R[X]$.
\end{proof}

\begin{prop}\label{Proposition 3.4} Let $R$ be a graded ring and $P$ be a proper graded ideal of $R$.
\begin{enumerate}
\item $R[X]$ has a graded strongly $1$-absorbing primary ideal if and only if $Grad(\{0_{R}\})$ is a graded prime ideal of $R$.

\item If $P[X]$ is a graded strongly $1$-absorbing primary ideal of $R[X]$, then $P$ is a graded strongly $1$-absorbing primary ideal of $R$.

\item The graded ideal $P + XR[X]$ is never a graded strongly $1$-absorbing primary ideal of $R[X]$.

\item $P[X]$ is a graded strongly $1$-absorbing primary ideal of $R[X]$ if and only if $P[X]$ is graded primary and $Grad(P)=Grad(\{0\})$.
\end{enumerate}
\end{prop}

\begin{proof}
\begin{enumerate}
\item From Theorem \ref{Theorem 2.6}, $R[X]$ has a graded strongly $1$-absorbing primary ideal if and only if $Grad(\{0_{R[X]}\})$ is graded prime (since $R[X]$ is never graded local). Moreover, $Grad(\{0_{R[X]}\})=(Grad(\{0_{R}\}))[X]$ and  $(Grad(\{0_{R}\}))[X]$ is a graded prime ideal of $R[X]$ if and only $Grad(\{0_{R}\})$ is a graded prime ideal of $R$.

\item If $P[X]$ is a graded strongly $1$-absorbing primary ideal of $R[X]$, then by Corollary \ref{Corollary
3.2 (2)}, $P = P[X]\bigcap R$ is a graded strongly $1$-absorbing primary ideal of $R$.

\item Note that $X=1.X\in R_{e}X\subseteq R_{e}[X]=\left(R[X]\right)_{e}$, and so $X\in h(R[X])$, and then by Lemma \ref{1}, $XR[X]$ is a graded ideal of $R[X]$, and then also by Lemma \ref{1}, $P+XR[X]$ is a graded ideal of $R[X]$. Since $R[X]$ is not graded local and $P+XR[X]\nsubseteq Grad(\{0_{R[X]}\})$ (because $X\notin Grad(\{0_{R[X]}\})$), $P + XR[X]$ is never a graded strongly $1$-absorbing primary ideal of $R[X]$.

\item Since $R[X]$ is not graded local, $P[X]$ is a graded strongly $1$-absorbing primary ideal of $R[X]$ if and only if $P[X]$ is graded $Grad(\{0_{R[X]}\})$-primary if and only if $P[X]$ is graded primary and $(Grad(\{0_{R}\}))[X]=Grad(\{0_{R[X]}\}) =Grad(P[X]) =(Grad(P))[X]$ if and only if $P[X]$ is graded primary and $Grad(\{0_{R}\}) =Grad(P)$.
\end{enumerate}
\end{proof}

\end{document}